\documentclass[12pt]{amsart}
\usepackage{xy}
\usepackage{bbm}
\usepackage{graphicx}
\usepackage{mathrsfs}
\usepackage{spectralsequences}
\usepackage{amsmath, amssymb, amsthm, latexsym}
\usepackage{amssymb,amscd,amsmath}
\usepackage{latexsym}
\usepackage{MnSymbol}
\usepackage{enumitem}
\usepackage[center]{caption}
\usepackage{tikz}
\newcounter{braid}
\newcounter{strands}

\DeclareMathAlphabet{\bsf}{OT1}{cmss}{bx}{n}

\pgfkeyssetvalue{/tikz/braid height}{1cm}
\pgfkeyssetvalue{/tikz/braid width}{1cm}
\pgfkeyssetvalue{/tikz/braid start}{(0,0)}
\pgfkeyssetvalue{/tikz/braid colour}{black}
\pgfkeys{/tikz/strands/.code={\setcounter{strands}{#1}}}

\makeatletter
\def\cross{%
  \@ifnextchar^{\message{Got sup}\cross@sup}{\cross@sub}}

\def\cross@sup^#1_#2{\render@cross{#2}{#1}}

\def\cross@sub_#1{\@ifnextchar^{\cross@@sub{#1}}{\render@cross{#1}{1}}}

\def\cross@@sub#1^#2{\render@cross{#1}{#2}}

\def\render@cross#1#2{
  \def\strand{#1}
  \def\crossing{#2}
  \pgfmathsetmacro{\cross@y}{-\value{braid}*\braid@h}
  \pgfmathtruncatemacro{\nextstrand}{#1+1}
  \foreach \thread in {1,...,\value{strands}}
  {
    \pgfmathsetmacro{\strand@x}{\thread * \braid@w}
    \ifnum\thread=\strand
    \pgfmathsetmacro{\over@x}{\strand * \braid@w + .5*(1 - \crossing) * \braid@w}
    \pgfmathsetmacro{\under@x}{\strand * \braid@w + .5*(1 + \crossing) * \braid@w}
    \draw[braid] \pgfkeysvalueof{/tikz/braid start} +(\under@x pt,\cross@y pt) to[out=-90,in=90] +(\over@x pt,\cross@y pt -\braid@h);
    \draw[braid] \pgfkeysvalueof{/tikz/braid start} +(\over@x pt,\cross@y pt) to[out=-90,in=90] +(\under@x pt,\cross@y pt -\braid@h);
    \else
    \ifnum\thread=\nextstrand
    \else
     \draw[braid] \pgfkeysvalueof{/tikz/braid start} ++(\strand@x pt,\cross@y pt) -- ++(0,-\braid@h);
    \fi
   \fi
  }
  \stepcounter{braid}
}

\tikzset{braid/.style={double=\pgfkeysvalueof{/tikz/braid colour},double distance=1pt,line width=2pt,white}}

\newcommand{\braid}[2][]{%
  \begingroup
  \pgfkeys{/tikz/strands=2}
  \tikzset{#1}
  \pgfkeysgetvalue{/tikz/braid width}{\braid@w}
  \pgfkeysgetvalue{/tikz/braid height}{\braid@h}
  \setcounter{braid}{0}
  \let\sigma=\cross
  #2
  \endgroup
}
\makeatother

\input xypic
\newtheorem{theorem}{Theorem}
\newtheorem{proposition}[theorem]{Proposition}

\newtheorem{lemma}[theorem]{Lemma}

\newtheorem{example}[theorem]{Example}


\def\Z{\mathbb{Z}}

\def\C{\mathbb{C}}

\def\C{\mathbb{C}}

\def\qed{\hfill$\square$\medskip}

\def\Zpk{\mathbb{Z}/p^{k}}
\def\Zpk1{\mathbb{Z}/p^{k-1}}

\newcommand{\rref}[1]{(\ref{#1})}

\newcommand{\beg}[2]{\begin{equation}\label{#1}#2\end{equation}}

\def\sl2{\widetilde{SL_{2}(\Z)}}

\title[Equivariant dihedral cohomology]{On the $RO(G)$-graded coefficients of dihedral equivariant cohomology}
\author{Igor Kriz and Yunze Lu}
\thanks{Kriz acknowledges the support of a Simons Collaboration Grant.}

\begin{document}
\maketitle

\begin{abstract}
We completely calculate the $RO(G)$-graded coefficients of ordinary equivariant cohomology where 
$G$ is the dihedral group of order $2p$ for a prime $p>2$ both with constant and Burnside ring coefficients. 
The authors first proved it for $p=3$ and then the second author generalized it to arbitrary $p$.
These are the first such calculations for a non-abelian group.
\end{abstract}

\section{Introduction}

A 1982 Northwestern conference problem asked for a complete calculation of the $RO(G)$-graded cohomology groups 
of a point for a non-trivial finite group $G$ (see \cite{mm} for definitions and \cite{lms} for background). 
This question was quickly solved by Stong \cite{stong} for cyclic groups
$\Z/p$ with $p$ prime. Partial calculations for groups $\Z/(p^n)$ and $(\Z/p)^n$ were much more recently done 
in \cite{hhr,hk, hok, hok2, sk}. In a recent lecture, Peter May \cite{maytalk} emphasized the fact that 
no case of a non-abelian group was known to date. 

\vspace{3mm}
The purpose of this paper is to advance progress in the non-abelian direction by calculating the $RO(G)$-graded cohomology coefficients for $G=D_{2p}$, dihedral group with $2p$ elements for $p$ a prime number, with both Burnside ring $\underline{A}$ and constant $\underline{\Z}$ coefficients. The constant mackey functor $\underline{\Z}$ is obtained by taking the quotient of the Burnside ring Mackey functor $\underline{A}$ by its augmentation ideal. Burnside Mackey functor is universal among among ordinary $RO(G)$-graded cohomology theories in the same sense as $\Z$-coefficients are non-equivariantly (see \cite{greenlees}), and thus were of primary interest historically. However, for non-trivial groups, the Burnside ring is not a regular ring, and because of that, passage from Burnside ring to other coefficients is not immediate. In applications \cite{hkreal, hhr}, the use of constant coefficients, which are simpler, prevailed so far.

\vspace{3mm}
Our main tool is using an explicit $D_{2p}$-equivariant CW structure on representation spheres, which will be described in the next section. We will state the calculation with constant $\underline{\Z}$ coefficients here, and postpone the statement with the Burnside ring coefficients till Section \ref{sburn} below, as it essentially follows from the constant case after some algebraic examinations of the Burnside rings. 

\vspace{3mm}
We present $G=D_{2p}$ as $$\{\zeta,\tau \, |\, \zeta^p=1,\tau^2=1,\zeta\tau=\tau\zeta^{-1}\}.$$ 
The group $G$ has two one-dimensional representations: the trivial representation denoted by $\epsilon$ and the sign representation denoted by $\alpha$. The group
$G$ also admits $\frac{p-1}{2}$ two-dimensional representations, denoted by $\gamma_i$'s, given by $$\gamma_i: \zeta \mapsto \left[
 \begin{matrix}
   \cos(\frac{2\pi i}{p}) & -\sin(\frac{2\pi i}{p}) \\
   \sin(\frac{2\pi i}{p}) & \cos(\frac{2\pi i}{p}) 
  \end{matrix}
  \right], \tau \mapsto \left[
 \begin{matrix}
   1 & 0 \\
   0 & -1 
  \end{matrix}
  \right], 1\leq i \leq \frac{p-1}{2}.$$ 
We will prove a periodicity result that will exempt 
us from distinguishing different two-dimensional representations. Hence the cohomology could be indexed by $k\epsilon+\ell\alpha+m\gamma$.

\vspace{3mm}
To discuss the $\underline{\Z}$-coefficient case, it is useful to recall the following calculation due to Stong 
(see \cite{lewis,hkreal}). 
Denote, for $\ell\geq 0$,
 \beg{ebl}{B_\ell=\widetilde{H}_{*}^{D_{2p}}(S^{\ell\alpha},\underline{\Z})
=\widetilde{H}_{*}^{\Z/2}(S^{\ell\alpha},\underline{\Z}),}
\beg{eblc}{B^\ell=\widetilde{H}^*_{D_{2p}}(S^{\ell\alpha},\underline{\Z})
=\widetilde{H}^*_{\Z/2}(S^{\ell\alpha},\underline{\Z}).}

\begin{proposition}\label{plb}
Let $n$ denote the grading. We have $$B_{\ell,n} =\left\{\begin{array}{ll}\Z & \text{$n=\ell$ even}\\
\Z/2 & \text{$0\leq n<\ell$ even}\\
0 & \text{else,}
\end{array}\right.$$
$$B^{\ell,n} =\left\{\begin{array}{ll}\Z & \text{$n=\ell$ even}\\
\Z/2 & \text{$3\leq n\leq\ell$ odd}\\
0 & \text{else.}
\end{array}\right. $$
\end{proposition}
\qed

\vspace{3mm}
We also put $$B_{\ell,n}=B^{-\ell,-n}, \; B^{\ell,n}=B_{-\ell,-n}\;\text{for $\ell<0$}.$$
Then \rref{ebl} and \rref{eblc} extend to $\ell<0$ by Spanier-Whitehead duality.

Now define ${}_sA_t$ and ${}^sA^t$ by
\beg{eal}{{}(_sA_t)_n=\left\{\begin{array}{ll}
\Z/p & \text{when $2s<n<2t-1$, $n\equiv 3\mod 4$},\\
0 & \text{else,}
\end{array}
\right.
}
\beg{eal2}{{}(^sA^t)^n=\left\{\begin{array}{ll}
\Z/p & \text{when $2s<n<2t-1$, $n\equiv 0\mod 4$},\\
0 & \text{else.}
\end{array}
\right.
}

Our main result, the $RO(G)$-graded (co)homology of a point with coefficients in $\underline{\Z}$ is given by the following

\begin{theorem} \label{t1} For $m>0$, we have
\beg{e1t1}{H^{D_{2p}}_{*}(S^{m\gamma+\ell\alpha},\underline{\Z})={}_{\ell-1}A_{\ell+m}[-\ell+1]
\oplus B_{\ell+m}[m],
}
\beg{e1t2}{H_{D_{2p}}^{*}(S^{m\gamma+\ell\alpha},\underline{\Z})={}^{\ell}A^{\ell+m}[-\ell+1]
\oplus B^{\ell+m}[m].
}
\end{theorem}

Here $[k]$ denotes shift up by $k$ in homology or cohomology. Note that since it is often appropriate to identify the cohomological grading with the nagative of homological, some authors prefer to define shifts in one grading only; in that case, there would be a negative sign in the square brackets of one of the formulas \rref{e1t1}, \rref{e1t2}. 

\vspace{3mm}

Theorem \ref{t1} and Proposition \ref{plb} give a complete calculation of the $RO(G)$-graded cohomology of a point with $\underline{\Z}$ coefficients. We will prove Theorem \ref{t1} in Section \ref{sp1}, \ref{sp2} below, and give the discussion of Burnside ring coefficients in Section \ref{sburn}.

\vspace{3mm}

\vspace{3mm}

\section{Equivariant CW-structure and periodicity}\label{sp1}

We will write $G=D_{2p}$ from now on. 
By abuse of notation, in addition to the generator of $D_{2p}$,
$\tau$ will also denote complex conjugation.
Also, we shall write $\gamma=\gamma_1$.

Let $S(m\gamma_i)$ be the unit sphere of the representation $m\gamma_i$. In this section we will construct a $D_{2p}$-equivariant CW structure on each $S(m\gamma_i)$. By computing the associated Mackey functor-valued equivariant chain complexes (meaning the Mackey functor-valued 
chain complexes given by the fixed points of the cellular chain complex
of the equivariant CW-complex with respect to subgroups) for different $\gamma_i$'s, we prove that instead of indexing on all $\gamma_i$'s, it suffices to consider only $\gamma$.

\vspace{3mm}

The CW structure is obtained by subdividing the standard $\mathbb{Z}/p$-equivariant cells of $S(m\gamma_i)$. We will identify the nonequivariant underlying spaces of all $S(m\gamma_i)$'s with subsets of $\mathbb{C}^m$
(by identufying each copy of $\gamma_i$ with a copy of $\C$). Then for $S(m\gamma_i)$, $\zeta\in G$ simply acts by coordinate-wise $\zeta^i_p$ multiplication where $\zeta_p=e^{2\pi i/p}$. In this context we will see $S(m\gamma_i)$'s share exactly the same CW decomposition non-equivariantly, with different $D_{2p}$-actions. 

First observe that the
usual free $\Z/p$-equivariant CW-sructure on $S(m\gamma_i)$ has equivariant
cells freely generated by the following non-equivariant cells for $1\leq k \leq m$: 
\beg{ez3c1}{\{(z_1,...,z_k,0,...,0)\in S(m\gamma_i)\mid z_k\in [0,1]\},
}
\beg{ez3c2}{\{(z_1,...,z_k,0,...,0)\in S(m\gamma_i)\mid z_k\in [0,1]\cdot e^{\lambda i},\;
(p-1)\pi/p\leq \lambda \leq (p+1)\pi/p\}.
}
Though both \rref{ez3c1} and \rref{ez3c2} are stable under the action of $\tau$, they are not $D_{2p}$-cells since $\tau$ acts non-trivially on them, and not all points of each corresponding open cell have
the same isotropy. 
However it is worth noting 
that they can be identified with unit disks of the representations \beg{ez3c3}{(k-1)\alpha+(k-1)\epsilon,\; k\alpha +(k-1)\epsilon,
}
respectively. This gives a guide on how to subdivide them into $D_{2p}$-equivariant cells. To be precise, we consider the following cells for $S(m\gamma_i)$:

\leftline{\textbf{Type A.}} $$a_{k,\ell}, \, 0\leq \ell \leq k-1, 1\leq k \leq m,$$ generated by $$\{(z_1,...,z_k,0,...,0) \in S(m\gamma_i) \,|\, \mathrm{Im}(z_\ell)\geq 0, z_{\ell+1},...,z_{k-1} \in [-1,1], z_k \in [0,1] \}.$$ The cell $a_{k,\ell}$ has dimension $k+\ell-1$  and has isotropy $\mathbb{Z}/2$ for $\ell=0$ and $\{e\}$ for $\ell>0$.

\leftline{\textbf{Type B.}} $$b_{k,\ell}, \, 0\leq \ell \leq k-1, 1\leq k \leq m,$$ generated by $$\{(z_1,...,z_k,0,...,0) \in S(m\gamma_i) \,|\, \mathrm{Im}(z_\ell)\geq 0, z_{\ell+1},...,z_{k-1} \in [-1,1], z_k \in [-1,0] \}.$$ 
Since it is symmetric to $a_{k,\ell}$, the cell $b_{k,\ell}$ has dimension $k+\ell-1$  and has isotropy $\mathbb{Z}/2$ for $\ell=0$ and $\{e\}$ for $\ell>0$. 

\leftline{\textbf{Type C.}} $$c_k, \, 1\leq k \leq m,$$ generated by $$\{(z_1,...,z_k,0,...,0) \in S(m\gamma_i) \,|\, z_k\in [0,1] \cdot e^{i\lambda}, 0\leq \lambda \leq \pi/p \}.$$ The cell $c_k$ has dimension $2k-1$  and has isotropy $\{e\}$.

It is straightforward to check that these cells give a $D_{2p}$-equivarant CW decomposition for each $S(m\gamma_i)$, only with different $D_{2p}$-actions for different $S(m\gamma_i)$'s. 

\vspace{3mm}

Based on the equivariant CW-structure, we are ready to write down the differentials. Note that 
the CW-structure is
regular: the boundaries of cells attach by homeomorphic embeddings, and 
hence the nonzero coefficients of the differentials
will always be either 1 or $-1$. We orient all cells as submanifolds (with corners) 
of the complex vector space $\mathbb{C}^m$. The induced orientation of the boundary of a cell is chosen by the following rule: the induced orientation followed by the outward normal direction together make up the standard orientation of 
$\C^m$. 
For example, the induced orientation of $S^1 \subset \mathbb{C}$ is going clockwise, hence the incidence number between $a_{2,1}$ and $c_1$ is $-1$. 

\begin{lemma}\label{lbred} Given $1\leq i\leq (p-1)/2$, let $1\leq j \leq p-1$ be the multiplicative inverse of $i$. Let $\zeta_i=\zeta^{j}$. With respect to the CW-structure and orientations described above, the $D_{2p}$-equivariant cell chain complex of $S(m\gamma_i)$ in the sense of Bredon \cite{bredon} has differential

$da_{1,0}=0$

$db_{1,0}=0$

$dc_1=\zeta_i^{\frac{p+1}{2}}b_{1,0}-a_{1,0}$

$da_{2,1}=-a_{2,0}-(1+\zeta_i+...+\zeta_i^{\frac{p-1}{2}})c_1+(\zeta_i+...+\zeta_i^{\frac{p-1}{2}})\tau c_1$

$db_{2,1}=-b_{2,0}-(1+\zeta_i+...+\zeta_i^{\frac{p-1}{2}})c_1+(\zeta_i+...+\zeta_i^{\frac{p-1}{2}})\tau c_1$

$da_{k,0}=a_{k-1,0}-b_{k-1,0} \qquad k>1$

$db_{k,0}=a_{k-1,0}-b_{k-1,0} \qquad k>1$

$da_{k,1}=a_{k-1,1}-b_{k-1,1}+(-1)^{k-1}a_{k,0} \qquad k>2$

$db_{k,1}=a_{k-1,1}-b_{k-1,1}+(-1)^{k-1}b_{k,0} \qquad k>2$

\vspace{2mm} 

\leftline{\rm{For} $k>3,\, 1<\ell<k-1$,}

$da_{k,\ell}=a_{k-1,\ell}-b_{k-1,\ell}+(-1)^{k-\ell}a_{k,\ell-1}+(-1)^{k-1}\tau a_{k,\ell-1}$

$db_{k,\ell}=a_{k-1,\ell}-b_{k-1,\ell}+(-1)^{k-\ell}b_{k,\ell-1}+(-1)^{k-1}\tau b_{k,\ell-1}$

\vspace{2mm}

\leftline{\rm{For} $k>2$, by abbreviating the action of $\sum_{j=1}^{(p-1)/2}\zeta_i^j$ to $\sigma$,}

$da_{k,k-1}=-a_{k,k-2}+(-1)^{k-1}\tau a_{k,k-2}-(1+\sigma)c_{k-1}+(-1)^{k-2}\sigma\tau c_{k-1}$

$db_{k,k-1}=-b_{k,k-2}+(-1)^{k-1}\tau b_{k,k-2}-(1+\sigma)c_{k-1}+(-1)^{k-2}\sigma\tau c_{k-1}$

\vspace{2mm}

\leftline{\rm{Finally, for} $k>1$, }

$dc_{k}=-a_{k,k-1}+(-1)^k\tau a_{k,k-1}+\zeta_i^{\frac{p+1}{2}}b_{k,k-1}+(-1)^{k-1}\zeta_i^{\frac{p+1}{2}}\tau b_{k,k-1}.$

\end{lemma}

\begin{proof}
We present here a computation for the differential of $a_{k,k-1}$ for $k>2$. By equivariance, it suffices to work on the generator, which is given by
$$\{(z_1,...,z_k,0,...,0) \in S(m\gamma_i) \,|\, \mathrm{Im}(z_{k-1})\geq 0,  z_k \in [0,1] \}.$$
Note that $z_k$ is uniquely determined by the values of $z_1,...,z_{k-1}$, and the dimension of the cell is $2k-2$. Hence we only need to consider cells of dimension $2k-3$ to which $a_{k,k-1}$ attaches. They are precisely those cells with $z_{k-1}$ coordinates lying on the boundary of $a_{k,k-1}$, namely, $$a_{k,k-2},\tau a_{k,k-2},\,c_{k-1},\zeta_i c_{k-1},...,\zeta_i^{(p-1)/2} c_{k-1} , \zeta_i \tau c_{k-1},...,\zeta_i^{(p-1)/2} \tau c_{k-1}.$$Here cells in the orbit of $c_{k-1}$ are those with $\text{Im} (z_{k-1})\geq 0$.

It remains to determine the incidence numbers between $a_{k,k-1}$ and these cells. By the rule set above, we could use
the basis
\beg{or1}{( e_1,ie_1,e_2,ie_2,...,e_{k-1},ie_{k-1})}
to determine the orientation of $a_{k,k-1}$, and the orientation of $\tau a_{k,k-2}$ could be described by
\beg{or2}{( e_1,-ie_1,e_2,-ie_2,...,e_{k-2},-ie_{k-2},e_{k-1}).}
On a point of $\tau a_{k,k-2}$ that $a_{k,k-1}$ attaches, the induced orientation is given by
\beg{or3}{( e_1,ie_1,...,e_{k-2},ie_{k-2},-e_{k-1})}
since juxtaposing with outward normal direction $-ie_{k-1}$ gives the same orientation as \rref{or1}. It is straightforward to compare orientations \rref{or2} and \rref{or3} and this gives the sign $$d a_{k,k-1} =...+ (-1)^{k-1}\tau a_{k,k-2}+...$$ in the formula. All the other computations follow by direct inspection in a similar way.
\end{proof}

\vspace{3mm}

Since $S^{m\gamma_i}$ is the unreduced suspension of $S(m\gamma_i)$, the $D_{2p}$-equivariant CW structure of $S^{m\gamma_i}$ is easily derived.

\vspace{3mm}

We will next prove that the choice of two-dimensional representation $\gamma_i$ doesn't matter in the computation
of ordinary equivariant cohomology. Let $\underline{A}$ denote the Burnside ring Green functor 
(see \cite{greenlees}).

\begin{proposition}\label{prop22} Let $\underline{M}$ be
a $D_{2p}$-Mackey functor.
The $D_{2p}$-stable homotopy type of the $H\underline{{A}}$-module spectrum $H\underline{{M}} \wedge S^{\gamma_i}$ does not depend on the choice of $i$.
\end{proposition}

The proof of this result will occupy
the remainder of this section. Let $\underline{M}$ be a Mackey functor. Generally, if $X$ is a finite $G$-CW complex, write $$X^n/X^{n-1}=X_{n+} \wedge S^n$$ where $X^n$ is the $n$th skeleton and $X_n$ is a discrete $G$-set. We have a chain complex of Mackey functors $\underline{C}_*(X;M)$ given by $$\underline{C}_n(X;\underline{M})=\underline{\pi}_0(H\underline{M}\wedge  X_{n+}).$$ 

It is also true that for any finite $G$-set $S$, $$\underline{C}_n(X;\underline{M})(S)=\underline{M}(S \times X_n),$$ which is the associated Mackey functor, also denoted by $\underline{M}_{X_n}$, to a finite $G$-set $X_n$.

\vspace{3mm}

To compute the $D_{2p}$-Mackey functor-valued chain complex 
$\underline{C}_*(S^{\gamma_i};\underline{M})$ for constant coefficient $\underline{\Z}$ and Burnside coefficient $\underline{A}$, we start with describing some $D_{2p}$-Mackey functors. Despite the
fact that the group $D_{2p}$ is non-abelian, its conjugacy relations among subgroups are simple and we can depict a $D_{2p}$-Mackey functor $\underline{M}$ by a diagram of the following form:

$$
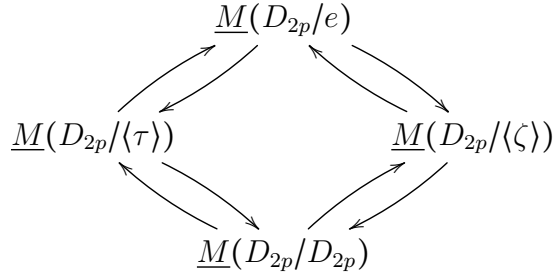
\xymatrix@C=1pt@R=25pt{
 & \underline{M}(D_{2p}/e) \ar@/^/[dr]\ar@/^/[dl] & \\
\underline{M}(D_{2p}/\langle\tau\rangle) \ar@/^/[ur]\ar@/^/[dr] &  & \underline{M}(D_{2p}/\langle\zeta\rangle) \ar@/^/[ul]\ar@/^/[dl]\\
 & \underline{M}(D_{2p}/D_{2p})\ar@/^/[ul]\ar@/^/[ur] &
}
$$

\begin{example}Constant Mackey functor $\underline{\mathbb{Z}}$.
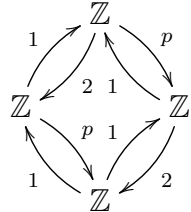
$$
\xymatrix@C=15pt@R=20pt{
 & \mathbb{Z} \ar@/^/[dr]^p\ar@/^/[dl]^2 & \\
\mathbb{Z} \ar@/^/[ur]^1\ar@/^/[dr]^p &  & \mathbb{Z} \ar@/^/[ul]^1\ar@/^/[dl]^2\\
 & \mathbb{Z}\ar@/^/[ul]^1\ar@/^/[ur]^1 &
}
$$
\end{example}
\begin{example}
Given a $\Z[G]$-module $M$, we have fixed-point Mackey functor $\underline{M}$ defined by $\underline{M}(G/H)=M^H$, restriction given by inclusion, and transfer given by summing over cosets. For example the fixed point Mackey functor $\underline{\mathbb{Z}[D_{2p}/\langle \tau\rangle]}$ is given by 
$$
\xymatrix@!C@C=10pt@R=20pt{
 & \mathbb{Z}[D_{2p}/\langle\tau\rangle] \ar@/^/[dr]^{(1,1,...,1)}\ar@/^/[dl]^B & \\
\frac{p-1}{2}\mathbb{Z}\oplus\mathbb{Z} \ar@/^/[ur]^A\ar@/^/[dr]^{(2,...,2,1)} &  & \mathbb{Z} \ar@/^/[ul]^{[1,1,...,1]}\ar@/^/[dl]^2\\
 & \mathbb{Z}\ar@/^/[ul]^{[1,1,...,1]}\ar@/^/[ur]^1 &
}
$$ Here round brackets stand for row vectors while square brackets stand for column vectors, and $$A=\left[
 \begin{matrix}
   0 & I_1\\
   I_{\frac{p-1}{2}} & 0\\
   J_{\frac{p-1}{2}} & 0
  \end{matrix}
  \right] ,B=\left[
 \begin{matrix}
   0 & I_{\frac{p-1}{2}} & J_{\frac{p-1}{2}}\\
   2 & 0 & 0
  \end{matrix}
  \right].$$
where $I_n$ is the $n\times n$ identity matrix and $J_n$ is the $n\times n$ minor diagonal identity matrix.

Similarly, the fixed point Mackey functor $\underline{\mathbb{Z}[D_{2p}/e]}$ is given by the following diagram.
$$
\xymatrix@!C@C=10pt@R=20pt{
 & \mathbb{Z}[D_{2p}/e] \ar@/^/[dr]^D\ar@/^/[dl]^{(I_p,I_p)} & \\
p\mathbb{Z} \ar@/^/[ur]^{[I_p,I_p]}\ar@/^/[dr]^{(1,1,...,1)} &  & \mathbb{Z}[D_{2p}/\langle\zeta\rangle] \ar@/^/[ul]^C\ar@/^/[dl]^{(1,1)}\\
 & \mathbb{Z}\ar@/^/[ul]^{[1,1,...,1]}\ar@/^/[ur]^{[1,1]} &
}
$$
where the matrices are represented by 
$$C=([1,...,1,0,...,0],[0,...,0,1,...,1]),$$ $$D=[(1,...,1,0,...,0),(0,...,0,1,...,1)].$$
\end{example}

The matrices above are derived by arranging the order of cells carefully: The basis of $\mathbb{Z}[D_{2p}/\langle\tau\rangle]$ can be identified with cells generated by $a_{1,0}$. Recalling 
that $\zeta_i$ acts by $2\pi/p$-rotation, we put a geometric counterclockwise order on
the cells 
$$a_{1,0},\zeta_i a_{1,0},...,\zeta_i^{p-1}a_{1,0}.$$ We also put an order on the
generators of $\mathbb{Z}[D_{2p}/\langle\tau\rangle]^{\langle\tau\rangle}$ by 
$$\zeta_i a_{1,0}+\zeta_i^{p-1}a_{1,0},...,\zeta_i^{\frac{p-1}{2}}a_{1,0}+\zeta_i^{\frac{p+1}{2}}a_{1,0},a_{1,0},$$ 
and this is why the upper left pair of arrows in the diagram for $\underline{\mathbb{Z}[D_{2p}/\langle \tau\rangle]}$ has the given matrix representation. 

The basis of $\mathbb{Z}[D_{2p}/e]$ can be identified with cells generated by $c_1$.  
We arrange them in the following order: $$c_1,\zeta_ic_1,...,\zeta_i^{p-1}c_1,\tau c_1,\tau\zeta_i c_1,...,\tau\zeta_i^{p-1}c_1.$$ The fixed point submodules are endowed with the induced order of basis.

\vspace{3mm}

Now fix $\underline{M}=\underline{\Z}$. In this case, by 
the double coset formula, the associated chain complex of Mackey functors can be calculated as fixed point Mackey functos. Hence using the examples above, the Mackey functor-valued $D_{2p}$-equivariant chain complexes for $S^{\gamma_i}$ is the following: $$\underline{\mathbb{Z}}\longleftarrow \underline{\mathbb{Z}[D_{2p}/\langle \tau\rangle]} \oplus \underline{\mathbb{Z}[D_{2p}/\langle \tau\rangle]} \longleftarrow \underline{\mathbb{Z}[D_{2p}/e]}.$$ The differentials are derived from Lemma ~\ref{lbred}. Since the differentials are $D_{2p}$-equivariant, we immediately see that all chain complexes for the different $S^{\gamma_i}$'s are isomorphic.

\vspace{3mm}

However, the isomorphism is not induced by any $D_{2p}$-equivariant map between the representation spheres. To prove Proposition \ref{prop22} we instead want to construct a functor $\mathscr{H}: Ch_{\geq 0}(Mack) \rightarrow D{\mathscr{S}p_G}$ such that 

(1). $\mathscr{H}\underline{M}=H\underline{M}$.

(2). $\mathscr{H}\underline{C}_*(X;\underline{M}) \simeq X \wedge H\underline{M}$.

\vspace{3mm}

\leftline{\textbf{Construction:}} Let $\mathscr{H}$ be the composition of the following functors $$Ch_{\geq 0}(Mack) \xrightarrow{K} sMack \xrightarrow{H} sD\mathscr{S}p_G \xrightarrow{|\cdot|} D\mathscr{S}p_G$$ where $K$ is the functor in Dold-Puppe correspondence which is an equivalence of first two categories; $H$ is the Eilenbeg-Maclane functor and $|\cdot|$ is geometric realization functor. The Eilenberg-Maclane construction is functorial; a recent account of this is in \cite{bohoso}. 

As an example we compute the case when $X=G/H_+$. Then $\underline{C}_*(X;\underline{M})$ is concentrated on degree 0. All the functors are computable, and we have $$\mathscr{H}\underline{C}_*(X;\underline{M}) = H\underline{M}_{G/H} \simeq H\underline{M} \wedge G/H_+.$$ The last equivalence can be verified by computing the homotopy groups of $H\underline{M} \wedge G/H_+$, and using the uniqueness of Eilenberg-MacLane spectra. 

In fact, one could make it into an natural isomorphism. By the projection formula $$G \ltimes_H H\underline{M} \cong G/H_+ \wedge H\underline{M}$$ and adjunction, it arises from the natural map of $H$-spectrum $H\underline{M} \rightarrow H\underline{M}_{G/H}$ induced by inclusion at the coset $eH$: $$\underline{M}\hookrightarrow \underline{M}_{G/H}.$$

\vspace{3mm}

For any finite $G$-CW complex $X$, we realize it as a simplicial $G$-set and the functor $\mathscr{H}$ is constructed as above. Then Proposition \ref{prop22} follows directly. 

\vspace{3mm}
\section{Proof of Theorem \ref{t1}}\label{sp2}

In this section we still focus on $\underline{\Z}$ Mackey functor coefficient and will present 
a proof of Theorem~\ref{t1}. To do this, first we calculate the $D_{2p}$-equivariant homology and cohomology of $S(m\gamma)$ as Bredon (co)homology. Recall that there is a cellular filtration on $S(m\gamma)$ by the $\Z/p$-equivariant cells generated by \rref{ez3c1}, \rref{ez3c2} of dimension $\leq s$. For $k\geq 1$, the filtration degree $2k-1$ part is generated by cells $b_{k,\ell}, c_k$ and the degree $2k-2$ part is generated by cells $a_{k,\ell}$. By using the differentials computed above, the corresponding homological spectral sequence has the following $E^1$-term: $$E_{2k-1,*}^1=B_{k-1}[k-1], \, \mathrm{for} \, 1\leq k \leq m$$ $$E_{2k,*}^1=B_k[k-1], \, \mathrm{for} \, 1\leq k \leq m.$$ The nontrivial differential $d^1$ is also determined by Lemma \ref{lbred}, which is an isomorphism except for $E_{4j,0}^1\rightarrow E_{4j-1,0}^1: \mathbb{Z}\xrightarrow{p} \mathbb{Z}$. On the two vertical edges $s=0,2m$, the terms also survive and the spectral sequence collapses to the $E^2$ page. In the case of cohomology, one just needs to turn subscripts into superscipts, mirror the computations by reversing arrows and use restriction maps of Mackey functors. Thus, we have proved the following

\begin{proposition}\label{prop31} For $m>0$, we have $$H^{D_{2p}}_*(S(m\gamma),\underline{\mathbb{Z}})=\mathbb{Z}\oplus {}_0{A}_m \oplus B_m[m-1],$$ $$H_{D_{2p}}^*(S(m\gamma),\underline{\mathbb{Z}})=\mathbb{Z}\oplus {}^0{A}^m \oplus B^m[m-1].$$
\end{proposition}
\qed

It may be tempting to try to use the same method for calculating the reduced $D_{2p}$-equivariant
(co)homology of $\Sigma^{\ell\alpha}\wedge S(m\gamma)_+$ for $\ell\in\Z$, but there are two difficulties. First, for $\ell>0$, the chain complex we obtain by smashing the CW-complexes cell-wise grows with $\ell$. More importantly, for $\ell<0$, the method actually fails: the Bredon chain complex is not an equivariantly stable object, and actually does not exist for spectra obtained by desuspending by non-trivial representations. There is, of course, a concept of an equivariant CW-spectrum \cite{lms}, but any chain complex in this stable context has to be built  directly on the Mackey functor level.

\vspace{3mm}

We proceed as follows: Suspend the filtration above by $S^{\ell\alpha}$. The corresponding spectral sequence's $d^1$ is determined by $S^{\ell\alpha}$-suspension of the connecting map $F_{2k+2}/F_{2k+1}\rightarrow \Sigma F_{2k+1}/F_{2k}$ of the triad 
$$(F_{2k+2},F_{2k+1},F_{2k}).$$ 
(Note: those are ``odd-to-even" connecting maps; by Lemma \ref{lbred}, the ``even-to-odd" connecting maps are $0$.)
Stably it does not depend on $\ell$. Note that the filtration quotients have the following form: $$F_{2k+1}/F_{2k}\cong D_{2p} \ltimes_{\mathbb{Z}/2} S^{k+(k+1)\alpha}, \quad F_{2k+2}/F_{2k+1}\cong D_{2p} \ltimes_{\mathbb{Z}/2} S^{(k+1)+(k+1)\alpha}.$$ Hence the connecting map is a stable 
$D_{2p}$-equivariant map $$D_{2p}/(\mathbb{Z}/2)_+ \rightarrow D_{2p}/(\mathbb{Z}/2)_+$$ 
By adjunction, it is equivalent to a 
$\mathbb{Z}/2$-equivariant stable map $$S^0 \rightarrow D_{2p}/(\mathbb{Z}/2)_+$$ which is classified by an element in 
\beg{eabbbb}{A(\mathbb{Z}/2) \oplus \mathbb{Z}^{\oplus (p-1)/2}}
by the Wirthm\"{u}ller isomorphism (which allows us to switch the source and target) and the fact that $\Z/2$-equivariantly, the orbit $D_{2p}/(\Z/2)$
is a disjoint union of one fixed point and $(p-1)/2$ free orbits.
The connecting map is read off from the attaching maps from $a_{k+1,k}$ to $c_k$, namely from
$$da_{k,k-1}=-a_{k,k-2}+(-1)^{k-1}\tau a_{k,k-2}-(1+\sigma)c_{k-1}+(-1)^{k-2}\sigma\tau c_{k-1}.$$
This shows that the connecting map does not depend on $k$, and is in \rref{eabbbb} represented by
the element
$$(1,1,\dots,1).$$
This corresponds to multiplication by $p$ on the constant Mackey functor $\underline{\Z}$.
It is convenient then to look at the spectral sequence of $\Sigma^{\ell+\ell\alpha}S(m\gamma)_+$, whose $E^1$ page is a shift of the conjunction of both cohomology and homology $E^1$ page for $S(m\gamma)$, and it also collapses to the $E^2$-page. Thus, we obtain 

\vspace{3mm}
\begin{proposition}\label{pmgl}
For $m>0$, we have
$$H_*^{D_{2p}}(\Sigma^{\ell\alpha}S(m\gamma)_+,\underline{\Z})=
B_\ell\oplus B_{\ell+m}[m-1] \oplus {}_\ell A_{\ell+m}[-\ell],$$
$$H^*_{D_{2p}}(\Sigma^{\ell\alpha}S(m\gamma)_+,\underline{\Z})=
B^\ell\oplus B^{\ell+m}[m-1] \oplus {}^\ell A^{\ell+m}[-\ell].$$
\end{proposition}
\qed

\vspace{3mm}

\noindent
{\em Proof of Theorem \ref{t1}:} We use the cofiber sequence $$\Sigma^{\ell\alpha}S(m\gamma)_+\rightarrow S^{\ell\alpha} \rightarrow S^{\ell\alpha+m\gamma}$$ to finish our computation. 
Looking at the long exact sequence in homology, the morphism $B_\ell \rightarrow B_\ell$ is the transfer map $p$, which is an isomorphism except in the top dimension when $\ell$ is even, and this gives an extra $\mathbb{Z}/p$. Besides, all the other components are shifted up by 1. Hence we have proved that $$H^{D_{2p}}_*(S^{\ell\alpha+m\gamma},\underline{\mathbb{Z}})=B_{\ell+m}[m] \oplus {}_{\ell-1}{A}_{\ell+m}[-\ell+1].$$ In cohomology the restriction maps always give isomorphisms, hence $$H_{D_{2p}}^*(S^{\ell\alpha+m\gamma},\underline{\mathbb{Z}})=B^{\ell+m}[m]\oplus {}^{\ell}{A}^{\ell+m}[-\ell+1].$$ 
\qed

\vspace{3mm}

\begin{example}

{\em As an example, we illustrate how to compute 
$$H^{D_{2p}}_*(\Sigma^{-4\alpha}S(5\gamma)_+,\underline{\Z}).$$ 
First we compute the $D_{2p}$-equivariant homology and cohomology of 
$S(5\gamma)$. The following is the $E^1$ page of the homological spectral sequence for $H^{D_{2p}}_*(S(5\gamma),\underline{\Z}).$

\vspace{10mm}

\begin{sseqdata}[ name = ss1, xscale = 1.1,
homological Serre grading, x label = {$E^1$ page for $H^{D_{2p}}_*(S(5\gamma),\underline{\Z})$ }, classes = {draw = none} ]

\class["\mathbb{Z}"](0,0)

\class["0"](1,0)

\class["0"](2,0)

\class["\Z"](3,0)

\class["\Z"](4,0)

\class["0"](5,0)

\class["0"](6,0)

\class["\Z"](7,0)

\class["\Z"](8,0)

\class["0"](9,0)

\class["\Z/2"](1,-1)

\class["\Z/2"](2,-1)

\class["\Z/2"](5,-1)

\class["\Z/2"](6,-1)

\class["\Z/2"](9,-1)

\class["\Z/2"](3,-2)

\class["\Z/2"](4,-2)

\class["\Z/2"](7,-2)

\class["\Z/2"](8,-2)

\class["\Z/2"](5,-3)

\class["\Z/2"](6,-3)

\class["\Z/2"](9,-3)

\class["\Z/2"](7,-4)

\class["\Z/2"](8,-4)

\class["\Z/2"](9,-5)

\d["p"]1(4,0)

\d["p"]1(8,0)

\d1(2,-1)

\d1(6,-1)

\d1(4,-2)

\d1(8,-2)

\d1(6,-3)

\d1(8,-4)

\end{sseqdata}

\printpage[ name = ss1, page = 1 ]

\vspace{10mm}

The differential $d^1$ is a multiplication by $p$ when there is a $\Z$ in the target (which is supported by $c_k$, $k$ even). The exception is filtration degree $2m-1=9$, where there is no
differential with that target, and filtration degree $0$, where there is no
differential with that source. There is no room for higher differentials
for dimensional reasons. Hence the spectral sequence collapses to the $E^2$ page. The two vertical edges and the $t=0$ line give the three summands in Proposition \ref{prop31}.

\vspace{5mm}

The following is the $E_1$ page of the cohomological spectral sequence. 

\vspace{10mm}

\begin{sseqdata}[ name = ss2, xscale = 1.1,
cohomological Serre grading, x label = {$E_1$ page for $H_{D_{2p}}^*(S(5\gamma),\underline{\Z})$}, classes = {draw = none} ]

\class["\mathbb{Z}"](0,0)

\class["0"](1,0)

\class["0"](2,0)

\class["\Z"](3,0)

\class["\Z"](4,0)

\class["\Z/2"](5,0)

\class["\Z/2"](6,0)

\class["\Z"](7,0)

\class["\Z"](8,0)

\class["\Z/2"](9,0)

\class["\Z/2"](7,-1)

\class["\Z/2"](8,-1)

\class["\Z/2"](9,-2)

\d["p"]1(3,0)

\d1(5,0)

\d["p"]1(7,0)

\d1(7,-1)

\end{sseqdata}

\printpage[ name = ss2, page = 1 ]

\vspace{10mm}

Now let us suspend by ${-4-4\alpha}$. Since the filtration on $S(5\gamma)$ is given by 
$$S^0,S^\alpha,S^{1+\alpha},...,S^{4+4\alpha},S^{4+5\alpha},$$
the filtered quotients are given by 
$$S^{-4-4\alpha},S^{-4-3\alpha},S^{-3-3\alpha},...,S^{-1},S^0,S^\alpha.$$
The following is the $E_1$ page, which is a shift of a juxtaposition of the dual of a truncation (at filtration degree 7) of the cohomological $E_1$ page and a truncation (at filtration degree 1) of the homological $E^1$ page.

\vspace{10mm}

\begin{sseqdata}[ name = ss3, xscale = 1.1,
homological Serre grading, x label = {$E^1$ page for $\tilde{H}^{D_{2p}}_*(\Sigma^{-4-4\alpha}S(5\gamma)_+,\underline{\Z})$ }, classes = {draw = none} ]

\class["0"](5,0)

\class["\Z"](4,0)

\class["\Z"](3,0)

\class["\Z/2"](2,0)

\class["0"](6,0)

\class["\mathbb{Z}"](8,0)

\class["\mathbb{Z}"](7,0)

\class["0"](9,0)

\class["\Z/2"](9,-1)

\class["\Z/2"](1,0)

\class["\Z"](0,0)

\class["\Z/2"](0,1)

\d["p"]1(8,0)

\d["p"]1(4,0)

\d1(2,0)

\end{sseqdata}

\printpage[ name = ss3, page = 1 ]

\begin{sseqdata}[ name = ss4, xscale = 1.1,
homological Serre grading, x label = {$E^2$ page for $\tilde{H}^{D_{2p}}_*(\Sigma^{-4-4\alpha}S(5\gamma)_+,\underline{\Z})$ }, classes = {draw = none} ]

\class["0"](5,0)

\class["0"](4,0)

\class["\Z/p"](3,0)

\class["0"](2,0)

\class["0"](6,0)

\class["0"](8,0)

\class["\mathbb{Z}/p"](7,0)

\class["0"](9,0)

\class["\Z/2"](9,-1)

\class["0"](1,0)

\class["\Z"](0,0)

\class["\Z/2"](0,1)

\end{sseqdata}

\printpage[ name = ss4, page = 1 ]

}
\end{example}

\vspace{3mm}

\vspace{5mm}

\section{Burnside ring coefficients}\label{sburn}

\vspace{3mm}

In this section we will apply the procedure above to deal with the Burnside ring coefficient $\underline{A}$. The difference is somewhat minor here due to the fact that most of the cells are free
and therefore the computation differs only
in a few dimensional degrees. 
We denote the Burnside rings additively by $A(\Z/2)=\Z\{1,t_2\}$, $A(\Z/p)=\Z\{1,t_p\}$, $A(D_{2p})=\Z\{1,t_2,t_p,t_{2p}\}$ where $t_i$ denotes the orbit of cardinality $i$. The Burnside Mackey functor $\underline{A}$ is depicted as
$$
\xymatrix@C=20pt@R=30pt{
 & \mathbb{Z} \ar@/^/[dr]^{1\mapsto t_p}\ar@/^/[dl]^{1\mapsto t_2} & \\
\Z\{1,t_2\} \ar@/^/[ur]^{t_2\mapsto 2}\ar@/^/[dr]^{1\mapsto t_p,t_2\mapsto t_{2p}} &  & \Z\{1,t_p\} \ar@/^/[ul]^{t_p \mapsto p}\ar@/^/[dl]^{(1,t_p)\mapsto (t_2,t_{2p})}\\
 & \mathbb{Z}\{1,t_2,t_p,t_{2p}\}\ar@/^/[ul]^{t_p \mapsto 1+\frac{p-1}{2}t_2,t_{2p} \mapsto pt_2}\ar@/^/[ur]^{t_{2p} \mapsto 2t_p} &
}
$$

\vspace{3mm}

Now we compare the associated Mackey functor-valued chain complex for different $S^{\gamma_i}$'s. By remembering the isotropy in the $\underline{\Z}$ case, we simply replace $\mathbb{Z}$ by the corresponding Burnside ring. For example the Mackey functor $\underline{A}_{D_{2p}/\langle \tau \rangle}$ is 
$$
\xymatrix@C=20pt@R=30pt{
 & p\Z \ar@/^/[dr]^{1\mapsto t_p}\ar@/^/[dl]^{1\mapsto t_2} & \\
A(\Z/2)\oplus \frac{p-1}{2}\Z \ar@/^/[ur]^{t_2\mapsto 2}\ar@/^/[dr]^{1\mapsto t_p,t_2\mapsto t_{2p}} &  & \Z\cong A(1) \ar@/^/[ul]^{t_p \mapsto p}\ar@/^/[dl]^{(1,t_p)\mapsto (t_2,t_{2p})}\\
 & A(\Z/2)\ar@/^/[ul]^{t_p \mapsto 1+\frac{p-1}{2}t_2,t_{2p} \mapsto pt_2}\ar@/^/[ur]^{t_{2p} \mapsto 2t_p} &
}
$$

\vspace{3mm}

Since all restriction and transfer maps in the associated Mackey functors are induced from the ones in $\underline{A}$, if we again order equivariant cells carefully, then the maps are the same. The differential is also induced from equivariant differential on the cellular level, so the chain complexes $$\underline{A} \longleftarrow \underline{A}_{D_{2p}/\langle \tau \rangle} \oplus \underline{A}_{D_{2p}/\langle \tau \rangle} \longleftarrow  \underline{A}_{D_{2p}/\langle e \rangle}$$are all equivalent, and we also have periodicity for the Burnside ring coefficients.

\vspace{3mm}

Denote, for $\ell \geq 0$,
\beg{ebbl}{\mathscr{B}_\ell=\widetilde{H}_{*}^{\Z/2}(S^{\ell\alpha},\underline{A}),}
\beg{ebblc}{\mathscr{B}^\ell
=\widetilde{H}^*_{\Z/2}(S^{\ell\alpha},\underline{A}).} Denote by $I_{\Z/2}$ the kernel of the restriction $A(\Z/2)\rightarrow A(1)$, and by $J_{\Z/2}$ the cokernel of the induction $A(1)\rightarrow A(\Z/2)$. Both $I_{\Z/2}$ and $J_{\Z/2}$ are clearly isomorphic to $\Z$. We can repeat the computations via spectral sequences to obtain the following 

\begin{proposition}\label{plbb} (Stong \cite{lewis,hkreal})
For $\ell\geq 0$, we have
$$\mathscr{B}_{\ell,n} =\left\{\begin{array}{ll}J_{\Z/2} & n=0\\
\Z & \text{$n=\ell$ even}\\
\Z/2 & \text{$0< n<\ell$ even}\\
0 & \text{else,}
\end{array}\right.$$
$$\mathscr{B}^{\ell,n} =\left\{\begin{array}{ll} I_{\Z/2} & n=0\\
\Z & \text{$n=\ell$ even}\\
\Z/2 & \text{$3\leq n\leq\ell$ odd}\\
0 & \text{else.}
\end{array}\right. $$
\end{proposition}
\qed
\vspace{3mm}

By setting $$\mathscr{B}_{\ell,n}=\mathscr{B}^{-\ell,-n}, \; \mathscr{B}^{\ell,n}=\mathscr{B}_{-\ell,-n}
\;\text{for $\ell<0$}$$ the result extends to $\ell <0$ by Spanier-Whitehead duality. 

Using Lemma \ref{lbred}, similar computations give the results parallel to Proposition \ref{prop31}.

\vspace{3mm}
\begin{proposition}\label{pmgg}
For $m>0$, we have
$$H_*^{D_{2p}}(S(m\gamma),\underline{A})=A(\Z/2) \oplus{}_0A_{m}\oplus \mathscr{B}_m[m-1],$$
$$H^*_{D_{2p}}(S(m\gamma),\underline{A})=A(\Z/2) \oplus {}^0A^{m}\oplus \mathscr{B}^m[m-1].$$
\end{proposition}
\qed

To compute suspensions by $\ell \alpha$, we need to see what the attaching fmap $$D_{2p}/(\mathbb{Z}/2)_+ \rightarrow D_{2p}/(\mathbb{Z}/2)_+$$ induces on $\underline{A}$. In terms of $\Z/2$-equivariant stable map $$S^0 \rightarrow D_{2p}/(\Z/2)_+,$$ it induces $$A(\Z/2)\rightarrow A(\Z/2)$$ $$1 \mapsto 1+\frac{p-1}{2}t_2, t_2\mapsto pt_2.$$ Thus, it is injective, and its cokernel is $\Z/p$, just as with $\underline{\Z}$ coefficients. Hence 
\begin{proposition}\label{prop43}
For $m>0$, we have
$$H_*^{D_{2p}}(\Sigma^{\ell\alpha}S(m\gamma)_+,\underline{A})=\mathscr{B}_\ell\oplus \mathscr{B}_{\ell+m}[m-1] \oplus {}_\ell A_{\ell+m}[-\ell],$$
$$H^*_{D_{2p}}(\Sigma^{\ell\alpha}S(m\gamma)_+,\underline{A})=\mathscr{B}^\ell\oplus \mathscr{B}^{\ell+m}[m-1] \oplus {}^\ell A^{\ell+m}[-\ell].$$
\end{proposition}
\qed
\vspace{3mm}

Finally, to get the (co)homology of $S^{\ell\alpha+m\gamma}$, we look at the cofiber sequence again. Denote 
\beg{ebcl}{\mathscr{C}_\ell=\widetilde{H}_{*}^{D_{2p}}(S^{\ell\alpha},\underline{A}),}
\beg{ebclc}{\mathscr{C}^\ell
=\widetilde{H}^*_{D_{2p}}(S^{\ell\alpha},\underline{A}).} Also denote by $J^{\Z/p}_{D_{2p}}$ the cokernel of 
the induction $A(\Z/p) \rightarrow A(D_{2p})$, and by $I^{\Z/p}_{D_{2p}}$ the kernel of the restriction
$A(D_{2p}) \rightarrow A(\Z/p)$. Both are isomorphic to $\Z\oplus\Z$ as groups. Since the Weyl group of $\Z/2\subset D_{2p}$ is itself and it acts trivially on the 
Burnside ring $A(\Z/p)$, spectral sequence computation shows that

\begin{proposition}\label{plc}
For $\ell\geq 0$, we have
$$\mathscr{C}_{\ell,n} =\left\{\begin{array}{ll}J_{D_{2p}}^{\Z/p} & n=0\\
A(\Z/p) & \text{$n=\ell$ even}\\
A(\Z/p)/2 & \text{$0< n<\ell$ even}\\
0 & \text{else,}
\end{array}\right.$$
$$\mathscr{C}^{\ell,n} =\left\{\begin{array}{ll} I^{\Z/p}_{D_{2p}} & n=0\\
A(\Z/p) & \text{$n=\ell$ even}\\
A(\Z/p)/2 & \text{$3\leq n\leq\ell$ odd}\\
0 & \text{else.}
\end{array}\right. $$
For $\ell<0$, put $$\mathscr{C}_{\ell,n}=\mathscr{C}^{-\ell,-n}, \mathscr{C}^{\ell,n}=\mathscr{C}_{-\ell,-n}.$$
\end{proposition}
\qed
\vspace{3mm}

Note that we have short exact sequences $$0\rightarrow \mathscr{B}_\ell \xrightarrow{ind^{\Z/2}_{D_{2p}}} \mathscr{C}_\ell \rightarrow \mathscr{B}_\ell\rightarrow 0,$$ and $$0\rightarrow \mathscr{B}^\ell \rightarrow \mathscr{C}^\ell \xrightarrow{res^{D_{2p}}_{\Z/2}} \mathscr{B}^\ell\rightarrow 0.$$ Therefore

\vspace{3mm}
\begin{theorem}\label{tg1}
For $m>0$, we have
$$H_{*}^{D_{2p}}(S^{m\gamma+\ell\alpha},\underline{A})=\mathscr{B}_\ell
\oplus{}_{\ell-1}A_{\ell+m}[-\ell+1]
\oplus \mathscr{B}_{\ell+m}[m],
$$
$$H^{*}_{D_{2p}}(S^{m\gamma+\ell\alpha},\underline{A})=
\mathscr{B}^\ell\oplus {}^{\ell}A^{\ell+m}[-\ell+1]
\oplus \mathscr{B}^{\ell+m}[m].
$$
\end{theorem}
\qed

\end{document}